\documentclass[oneside, 12pt]{amsart}
\usepackage[a4paper, top=3cm, left=2.8cm, right=2.8cm]{geometry}
\usepackage{amscd, amssymb, amsmath, mathrsfs}
\usepackage[english]{babel}
\usepackage{url}
\usepackage{tikz-cd}
\usepackage{booktabs}
\usepackage[pdftex, colorlinks=true,  citecolor=blue, linkcolor=blue, linktocpage=true]{hyperref}

\DeclareUrlCommand\arXiv{\urlstyle{same}}

\newcommand\mylabel[1]{\label{#1}\marginpar{\vspace{-1ex}\medskip\medskip\footnotesize \tt #1}}
\renewcommand\mylabel[1]{\label{#1}}
\newcommand{\mydate}{
\number\day\space
\ifcase\month \or January\or February\or March\or April\or May\or June\or July\or August\or September\or October\or November\or December\fi 
\space\number\year}

\newtheorem{theorem}{Theorem}[section]
\newtheorem{maintheorem}{Theorem}
\newtheorem{lemma}[theorem]{Lemma}
\newtheorem{proposition}[theorem]{Proposition}

\theoremstyle{definition}

\newtheorem*{acknowledgement}{Acknowledgement}

\theoremstyle{remark}

\newcommand{\ZZ}{\mathbb{Z}}

\newcommand{\ideala}{\mathfrak{a}}

\newcommand{\Frac}{\operatorname{Frac}}

\newcommand{\Kernel}{\operatorname{Ker}}

\newcommand{\dirlim}{\varinjlim}
\newcommand{\invlim}{\varprojlim}

\newcommand{\lra}{\longrightarrow}

\newcommand{\maxid}{\mathfrak{m}}

\newcommand{\primid}{\mathfrak{p}}

\newcommand{\pr}{\operatorname{pr}}

\newcommand{\quadand}{\quad\text{and}\quad}

\newcommand{\ra}{\rightarrow}

\newcommand{\Rad}{\operatorname{Rad}}

\newcommand{\Spec}{\operatorname{Spec}}

\newcommand{\topO}{\mathfrak{O}}

\newcommand{\qc}{{\text{\rm qc}}}

\newcommand{\Frk}{\operatorname{Frk}}
\newcommand{\Jac}{\operatorname{Jac}}
\newcommand{\kol}{{\text{\rm kol}}}
\newcommand{\sprl}{{\text{\rm sprl}}}
\newcommand{\frk}{{\text{\rm frk}}}
\newcommand{\dsc}{{\text{\rm dsc}}} 

\begin{document}

\title[Hochster's Theorem]
      {A simple proof for  Hochster's Theorem}

\author[Stefan Schr\"oer]{Stefan Schr\"oer}
\address{Heinrich Heine University D\"usseldorf, Faculty of Mathematics and Natural Sciences, Mathematical Institute, 40204 D\"usseldorf, Germany}
\curraddr{}
\email{schroeer@math.uni-duesseldorf.de}

% 14A05 Relevant commutative algebra 
% 13A18 Valuations and their generalizations for commutative rings
% 54F05 Linearly ordered topological spaces, generalized ordered spaces, and partially ordered spaces
% 54A10 Several topologies on one set (change of topology, comparison of topologies, lattices of topologies)
% 54D10 Lower separation axioms (T0-T3 , etc.)

%Possible journals:
%Geometry & Topoology: Dan Abramovic, Sandor Kovacs
%PAMS Claudiu Raicu, Gregory G. Smith

\dedicatory{16 June 2026}

\begin{abstract}
We give a   conceptual proof for Hochster's Theorem, which asserts
that each spectral space is homeomorphic to the spectrum of a ring. Given a ground field and a spectral space, our ring is constructed
as   filtered direct limit of prime-finite ring, which are attached in a functorial way to finite Kolmogoroff spaces.
The construction simplifies an argument of Ershov along these lines. Our crucial ingredient is an assembly of finite Kolmogoroff
spaces in terms of coequalizers and pushouts of  one-dimensional spaces, and Schwede's observation on prime ideals in cartesian squares of rings.
\end{abstract}

\maketitle
\tableofcontents

%===========================================================
\section*{Introduction}
\mylabel{Introduction}

For a   ring $R$, the Zariski topology $\topO$ for the space  $\Spec(R)$ of prime ideals is generated by the basic open sets
$D(f)=\{\primid\mid f\not\in \primid\}$, where $f $ ranges over the ring elements. This topology   has    two intrinsic   properties: 
(i) The family of quasicompact open sets $U_\lambda$, $\lambda\in L$ is stable under finite intersections, and forms a basis. 
(ii) Each irreducible closed set $Z$ has exactly one generic point.
Note that that first condition implies that the space  is quasicompact, by taking the empty intersection.
Also  note that  spaces satisfying the second condition are called \emph{sober spaces}.
These are instances of  \emph{Kolmogoroff spaces}, where for each pair of points $a\neq b$ there is an open set
that contains one but not the other point.

The  topological spaces  $X$ satisfying both (i) and (ii) is called   \emph{spectral spaces}.
They play an important role in the theory of locales and topoi (\cite{Johnstone 1982} and  \cite{Picado; Pultr 2012}),
and the recent monograph of Dickmann, Schwartz and Tressl
\cite{Dickmann; Schartz; Tressl 2019}   is entirely devoted to them.
Hochster's famous theorem  \cite{Hochster 1969} asserts that   \emph{each spectral space is homeomorphic to the spectrum of a ring}.
This was established by an ingenious  application of valuation theory, but  one may fairly say that Hochster's   proof is   long and complicated. Moreover,
the construction  depends on choices, and does not seem to depend on the space   in a functorial way.

The finite spectral spaces   $X$ are precisely the finite  Kolmogoroff spaces.
Somewhat surprisingly, these have     important applications in homotopy theory 
(\cite{McCord 1966}, \cite{Stong 1966}, \cite{Barmak 2011}, \cite{May; Pishevar 2026}).
Using certain cartesian squares of rings, Lewis \cite{Lewis 1973}  showed that each finite Kolmogoroff space $X$ 
is homeomorphic to the spectrum of a prime-finite ring $R$, and explicitly raises the question whether one can make $R$ functorial in $X$.
This indeed suggests a more conceptual and intrinsic  approach to Hochster's Theorem, because
the spectral spaces   are precisely the filtered inverse limits  
of finite Kolmogoroff spaces.  If one could realize
such a system $(X_\lambda)_{\lambda\in L}$  
with a filtered direct system  $(R_\mu)_{\lambda\in L}$ of prime-finite rings,  
the inverse limit of spaces  would be obtained from  the direct limit of rings.   
This idea was indeed carried out by Ershov \cite{Ershov 2005},
but the beauty and simplicity of the reasoning is somewhat buried.

The goal of this paper is to give a  simpler and clearer version of  Erschov's arguments. 
Our approach is based on assembling   finite Kolmogoroff spaces in an intrinsic  way via finite Kolmogoroff spaces
of dimension at most one, together with a beautiful  
observation of Schwede \cite{Schwede 2005} on prime ideals in cartesian square of rings.
Given a finite Kolmogoroff space $X$ and a ground field $k$, we construct the ring $R(X)$ whose spectrum recovers the space
as a certain subring 
\begin{equation}
\label{eq:inclusion in product}
R(X)\subset \prod_{x\in X} k(T_U)_{U\ni x}.
\end{equation} 
Here the factors on the right are purely transcendental field extensions, where the indeterminates $T_U$ correspond
to the open neighborhoods $U$ of the point $x$. The entries  $(P_x)_{x\in X}$ for the ring elements in  $R(X)$ belong
to certain semilocal Dedekind domains with field of fractions  $k(T_U)_{U\ni x}$, 
and their classes modulo the Jacobson radical is related to the entries $P_\sigma$, where $\sigma\in\overline{\{x\}}$,
as detailed in Section \ref{sec:functorial construction}.

I find the construction somewhat analogous to \emph{toric varieties}, where one attaches to the combinatorial datum of a fan $(N,\Delta)$,
together with the choice of a ground field $k$, a collection of rings $k[N\cap \sigma^\vee]$ indexed by the cones $\sigma\in\Delta$,
resulting in the toric variety  $ \operatorname{Temb}_N(\Delta)=\bigcup_\sigma \Spec k[N\cap \sigma^\vee]$.
Another analogy is the construction of \emph{Chevalley groups} from root data. 
In our setting,  the combinatorial datum is a finite ordered set $(|X|,\preccurlyeq)$, or equivalently a  finite Kolmogoroff space $(|X|,\topO)$,
and the output is a prime-finite  ring $R(X)$.

The kernels for the projections in \eqref{eq:inclusion in product}  define a   map $X\ra\Spec(R(X))$. 
The  construction is canonical and does not involve choices,  and our  main result it:

\begin{maintheorem}
(see Thm.\ \ref{thm:natural homeomorphism})
The maps $X\ra \Spec(R(X))$ are homeomorphisms, which are natural with respect to continuous surjections   of finite Kolmogoroff spaces.
\end{maintheorem}

Since each spectral space $X$ is  a filtered inverse limit of finite Kolmogoroff spaces $X_\lambda$
with surjective transition maps,   Hochster's theorem is an immediate corollary.

We take   opportunity to shed additional light on these inverse limits from a down-to-earth perspective 
(avoiding the patch topology and locales): Given an arbitrary topological space 
 $X=(|X|,\topO)$ we  consider  the spectral space $X^\sprl=\invlim X_\lambda$,
where the inverse limit runs over all \emph{finite topologies} $\topO_\lambda\subset\topO$, and the   $X_\lambda$ are the Kolmogoroffizations.
For quasicompact $X$, one can project onto  $X^\sprl_\qc=\invlim X_\mu$, where all $\topO_\mu$-open sets are $\topO$-quasicompact.
This has a certain universality property:

\begin{maintheorem}
(see Thm.\ \ref{thm:universal map})
For each quasicompact  space $X$,  the canonical map  $X\ra X^\sprl_\qc$ is universal for quasicompact continuous maps
to spectral spaces.
\end{maintheorem}

Note that for an arbitrary scheme $S$, 
the underlying topological space $X=|S|$ is spectral  if and only if   the scheme  is quasicompact and quasiseparated.
Also note  that each morphism $f:\Spec(R')\ra \Spec(R)$ of affine schemes, the underlying  continuous map   is quasicompact.
This already suggests that in the category of spectral spaces, the quasicompact continuous maps play an important role,
as already observed, for example, in \cite{Banaschewski 1996} and \cite{Schwartz 2013}.

\medskip
The paper is organized as follows: In Section \ref{sec:assembly} we analyze how finite Kolmogoroff spaces arise
via coequalizers and pushouts from finite Kolmogoroff spaces of dimension at most one.
In Section \ref{sec:functorial construction} we attach in a functorial way 
to each finite Kolmogoroff space $X$ a  prime-finite  ring $R(X)$ whose spectrum
recovers the space. In Section \ref{sec:spectralization} we discuss how spectral spaces arise as inverse limits of finite Kolmogoroff spaces,
and how Hochster's Theorem follows from this.

\begin{acknowledgement}
The research  was  also conducted       in the framework of the   research training group
\emph{GRK 2240: Algebro-Geometric Methods in Algebra, Arithmetic and Topology}.
\end{acknowledgement}

 %===========================================================
\section{Assembly for finite Kolmogoroff spaces}
\mylabel{sec:assembly}

Recall that a topological space $X$ is called \emph{Kolmogoroff} if for each pair of points $a\neq b$
there is an open set $U$ that contains one but not the other point. This is exactly the class of spaces
for which one may reconstruct the topological space   from the lattice of open sets, and thus from 
the topos of sheaves.  
Note that  the finite Kolmogoroff spaces $(|X|,\topO)$ correspond to the finite ordered sets
$(|X|,\preccurlyeq)$, where the order relation $a\preccurlyeq b$ corresponds to  $a\in \overline{\{b\}}$.
In this section we are  interested in the finite Kolmogoroff spaces $X$.

The finite Kolmogoroff spaces $X$ that are irreducible of dimension at most one take the form $X=\{\eta,\sigma_1,\ldots,\sigma_r\}$
with some $r\geq 0$, and the non-empty open sets are precisely the sets containing the generic point $\eta$.
In turn,  every finite set $X$ with a distinguished point $\eta$ can be viewed as such a Kolmogoroff space.
Let us call such spaces \emph{forks}.
Note that   singletons counts as forks, and that the \emph{pointed forks} $X^*=X\smallsetminus\{\eta\}$ are discrete.
Also note  that each fork is homeomorphic to the spectrum  of either  a suitable  semilocal Dedekind domain, or a  field.

Let $X$ be a finite Kolmogoroff space. Given a point $x\in X$, we define
$$
\Frk(x)=\overline{\{x\}}=\{x,\sigma_1,\ldots,\sigma_r\}
$$
and endow it with the topology that turns the set into  fork with generic point $x$. Note that the  canonical inclusion $\Frk(x)\ra X$ is continuous, but not necessarily an embedding.
Consider the diagram
\begin{equation}
\label{eq:coequalizer diagram}
\begin{tikzcd} 
\dot\bigcup_{x\in X}\Frk^*(x) \ar[r,shift right=0.5ex]\ar[r,shift left=0.5ex]	& \dot\bigcup_{x\in X}\Frk(x) \ar[r] &	X,  
\end{tikzcd}
\end{equation} 
where the upper arrow sends    $\sigma\in\Frk^*(x)$ to the closed point $\sigma\in\Frk(x)$,
the lower arrow   sends $\sigma$ to the generic point in $\Frk(\sigma)$,
and the  map on the right stems from the inclusion of forks.

\begin{lemma}
\mylabel{lem:coequalizer}
For each finite Kolmogoroff space $X$, the  above   is a coequalizer diagram in the category of topological spaces.
\end{lemma}

\proof
Write  $f:\dot\bigcup_{x\in X}\Frk(x)\ra X$ for the continuous mapping in question,
and let $g:\dot\bigcup_{x\in X}\Frk(x)\ra  Y$ be a continuous map  to another topological space $Y$ where the  two restrictions
to $\dot\bigcup_{x\in X}\Frk^*(x)$ coincide,  and set $g_x=g|\Frk(x)$. 
We have to show that there is a unique continuous map $h:X\ra Y$ with $g=h\circ f$. Uniqueness   follows from surjectivity of $f$.
For existence, we are forced to set $h(\sigma)=g_x(\sigma)$, where  $x\in X$ is a chosen point containing $\sigma$ in its closure. This is indeed well-defined:
Viewing $\sigma$ as a generic point in $\Frk(\sigma)$ and as closed point in   $\Frk(x)$ we see $g_x(\sigma)=g_\sigma(\sigma)$, which does not depend on $x$.

It remains to check continuity: Let $U=h^{-1}(V)$ be the preimage of an open set $V\subset Y$. Then $f^{-1}(U)$ is open,   the two intersections with
$\dot\bigcup_{x\in X}\Frk^*(x)$ coincide, and the task is to show that $U$ is open. For our finite Kolmogoroff space $X$, this means for
$U$ is stable under generization. Fix $\sigma\in U$, and let $x\in X$ be a point with $\sigma\in\overline{\{x\}}$.
Then the open set $\Frk(x)\cap f^{-1}(U)$ contains $\sigma$, hence also $x$, and thus $x\in U$.
\qed

\medskip
This expresses a given finite Kolmogoroff space $X$ in terms of such spaces of dimension at most one.
The latter immediately arise as spectra from  semilocal Dedekind domains and field products, 
which already explains our strategy to construct prime-finite rings. 
To carry  this out we need further information that allows induction on dimension:

Suppose the finite Kolmogoroff space $X$ is non-empty,   
write $X^0=\{\eta_1,\ldots,\eta_n\}$ for the open set of generic points, and  
$X'=X\smallsetminus X^0$ for the complementary closed set. The former is discrete, and the latter
has $\dim(X')<\dim(X)$. Consider the commutative diagram
\begin{equation}
\label{eq:pushout kolmogoroff}
\begin{CD}
\dot\bigcup_{\eta\in X^0}\Frk^*(\eta)	@>>>	\dot\bigcup_{\eta\in X^0}\Frk(\eta) \\
@VVV						@VVV\\
X'					@>>>	X
\end{CD}
\end{equation} 
where all maps stem from the canonical inclusions. 

\begin{lemma}
\mylabel{lem:cocartesian}
For each non-empty finite Kolmogoroff space $X$, the above diagram is cocartesian in the category of topological spaces.
\end{lemma}

\proof
Set $W=\dot\bigcup_{\eta\in X^0}\Frk(\eta)$ and $W^*=\dot\bigcup_{\eta\in X^0}\Frk^*(\eta)$, form the amalgamated sum $P=X'\cup_{W^*}W$,
and write $f:P\ra X$ for the   continuous map defined by the diagram \eqref{eq:pushout kolmogoroff}.
The upper horizontal map is an inclusion, and its complement is the set of generic points $X^0\subset X$.
We thus have a  disjoint union $P=X'\cup X^0$ of sets, and conclude that the   map $f:P\ra X$ is bijective.
It remains to verify that this continuous map is open. Let $U\subset X$ be a subset
whose preimages on  $X'$ and   $W$ are open.
Fix some $u\in U$ and let $x\in X$ be a generization.
If $x=\eta$ is a generic point we have $x\in U$ by considering the preimage on $W$. If $x$ is non-generic we likewise
have $x\in U$, by looking at the intersection with $X'$, and it follows that $U\subset X$ is open.
\qed

%===========================================================
\section{Functorial construction of prime-finite rings}
\mylabel{sec:functorial construction}

Fix a ground field $k$. The goal of this section is to construct  for each finite Kolmogoroff space $X$
a   ring $R(X)$  together with a     homeomorphisms $X\ra \Spec(R(X))$. 
The crucial point of the construction is that  the ring  must be natural in the space, at least with respect to continuous surjections.
A naive choice is  the ring of all functions $X\ra K$ with  values in some field extension $K$.
This ring is obviously functorial, but fails to reflect the topology on $X$. 
 
To achieve our goals     choose  for each open set $U$  a formal symbol $T_U$.
Given a point $x\in X$, we write $k[T_U]_{U\ni x}$ for the
polynomial ring  whose  indeterminates  correspond to the open neighborhoods  of the   point at hand.
This yields a  ring of finite type
$$
A(X) =\prod_{x\in X} k[T_U]_{U\ni x}.
$$
The construction  is functorial:
Given any continuous map  $f:X\ra Y$   of finite Kolmogoroff spaces we obtain for each $x\in X$ a ring homomorphism  
\begin{equation}
\label{eq:induced map}
A(Y)=\prod_{y\in Y} k[T_V]_{V\ni y}\stackrel{\pr_{f(x)}}{\lra} k[T_V]_{V\ni f(x)} \lra k[T_U]_{U\ni x},
\end{equation} 
where the map on the right sends the indeterminate $T_V$ to the indeterminate $T_{f^{-1}(V)}$. By the universal property of products, 
this defines the desired ring homomorphism
$$
f^*:A(Y)\lra A(X),
$$
which turn $X\mapsto A(X)$ into a functor.

Next recall that a ring element is called \emph{regular},  if the corresponding homothety is injective (the term ``non-zero divisor''
is also in frequent usage).
For the ring $A(X)$, an element $(P_x)_{x\in X}$ is regular if and only if the entries $P_x$ are  non-zero. 

\begin{proposition}
\mylabel{prop:regular to regular}
If the continuous map $f:X\ra Y$ is surjective,  the ring homomorphism $f^*:A(Y)\ra A(X)$ sends regular elements to regular elements.
\end{proposition}

\proof
Let $(Q_y)_{y\in Y}\in A(Y)$ be an elements whose image $(P_x)_{x\in X}\in A(X)$ has vanishing entry $P_a$  for some $a\in X$.
From \eqref{eq:induced map} we see that this entry depends only on $Q_b$, where $b=f(a)$. Let $V_i\subset Y$, $1\leq i\leq r$ be the open neighborhoods of $b\in Y$.
Their preimage $U_i=f^{-1}(V_i)$ are pairwise different, because $f:X\ra Y$ is surjective. 
In turn, the elements $T_{U_i}\in k[T_U]_{U\ni x}$ are algebraically independent. Since $P_a=Q_b(T_{U_1},\ldots,T_{U_r})$ vanishes, the same
holds for $Q_b=Q_b(T_{V_1},\ldots,T_{V_r})$.  So if the image  $(P_x)_{x\in X}$ is non-regular, the same holds for 
the  argument $(Q_y)_{y\in Y}$.
\qed

\medskip
Passing to localizations with respect to the multiplicative system $S(X)$ of all regular elements, we see that the formation of the field products
$$
F(X)=\prod_{x\in X} k(T_U)_{U\ni x} = S(X)^{-1}A(X)
$$
is functorial with respect to continuous surjections. Its spectrum gives back the \emph{discretization} $X^\text{dsc}$, where every
subset of the underlying topological space $|X|$   is   open.

Given $x\in X$, we consider     $\Frk(x)=\{x,\sigma_1,\ldots,\sigma_r\}$, whose underlying set is $\overline{\{x\}}$.
For each $\sigma=\sigma_i$ we chosen 
an open set $U_x$ that contains $x$ but not $\sigma$. Following Ershov (\cite{Ershov 2005}, Section 2)    
we consider inside   $k(T_U)_{U\ni x}$ the discrete valuation ring
$$
R^x_\sigma = S^{-1}_{x,\sigma} \left(k(T_U,\frac{T_{U'}}{T_{U''}}) [T_{U_x}] \right),
$$
where $U$ ranges over all open sets containing $\sigma$, while $U',U''$ run over all open sets
containing $x$ but not $\sigma$, 
and $S_{x,\sigma}$ is the complementary multiplicative system for  the maximal ideal $\maxid_{x,\sigma}=(T_{U_x})$.
From $T_{U'}=\frac{T_{U'}}{T_{U''}}\cdot T_{U_x}$ with $U''=U_x$ we see $\Frac(R^x_\sigma)=k(T_U)_{U\ni x}$.
Also note that the above  definition does not depend on the choice of  $U_x$, because for any other such   $\tilde{U}_x$ we have 
$T_{\tilde{U}_x} =\frac{T_{V'}}{T_{V''}} \cdot  T_{U_x}$, with the open sets  $V'=U_x$ and $V''=\tilde{U}_x$.
Of course, one may choose $U_x$ as the smallest open neighborhood of $x$, but the extra flexibility will be useful in what follows.
We also remark in passing that $R^x_\sigma\subset k(T_U)_{U\ni x}$ 
appear  as localizations of  some blow-up rings of $k[T_U]_{U\ni x}$. It might be interesting
to pursue this line of thought.

We now can form the   the intersection ring
$$
R^x= R^x_{\sigma_1}\cap \ldots\cap R^x_{\sigma_r}.
$$
If $x\in X$ is already a closed point, the index set is  empty, and the above has to be interpreted as the field $k(T_U)_{U\ni x}$.
Otherwise,  the intersection ring  is   semilocal, with maximal ideals $\maxid_i= \maxid_{R^x_{\sigma_i}} \cap R^x$,
having $(R^x)_{\maxid_i}=R^x_{\sigma_i}$ and thus $\Frac(R^x)=k(T_U)_{U\ni x}$, as explained in \cite{AC 5-6}, Chapter V, \S 7, No.\ 1.
One infers with the Nakayama Lemma that every ideal $\ideala\subset R^x$ is finitely generated, hence $R^x$ is a semilocal Dedekind domain,
and in particular \emph{prime-finite}, meaning that the rings contains only finitely many prime ideals.
The  quotient by the Jacobson radical takes the form
\begin{equation}
\label{eq:modulo jacobson}
\overline{R^x}= R^x/\Jac(R^x)    =\prod_{\sigma\in\Frk^*(x)}  k(T_U,\frac{T_{U'}}{T_{U''}}).
\end{equation} 
Note that in each factor on the right,
the collection of open sets $U,U',U''$ depends on the index $\sigma$, although this is not reflected by our  notation.
Also note that the  spectrum of  $\prod_{x\in X} R^x$ gives back   $X^\frk=\dot\bigcup_{x\in X}\Frk(x)$, which one may call the
\emph{forkification}.

We now exploit that \eqref{eq:modulo jacobson} contains 
$\prod_{\sigma\in\Frk^*(x)} k(T_U)_{U\ni \sigma}$ as a subring,  and therefore also 
the $\prod_{\sigma\in\Frk^*(x)} R^\sigma$. This   allows us to form the subring
$$
R(X)\subset   \prod_{x\in X} R^x
$$
of all tuples $(P_x)_{x\in X}$ such that the entries for all the non-closed $x\in X$ satisfies the congruence  condition
\begin{equation}
\label{eq:condition subring}
\overline{P_x}  =  (P_{\sigma_1},\ldots,P_{\sigma_r}), 
\end{equation} 
where the left hand side is the residue class  in  $\overline{R^x}=R^x/\Jac(R^x)$,
and the tuple on the right is indexed by the   points  $\sigma_i\in\Frk^*(x)$. The subring  keeps the desired  functoriality properties:

\begin{proposition}
\mylabel{prop:respects subrings}
For each continuous surjection $f:X\ra Y$ of finite Kolmogoroff spaces, the induced homomorphism $f^*:F(Y)\ra F(X)$ sends
the subring  $R(Y)$ to the  subring  $R(X)$.
\end{proposition}

\proof
We first check that the homomorphism $f^*$ sends the product ring $\prod_{y\in Y}R^y$ to the product ring $\prod_{x\in X}R^x$.
Let $(P_x)_{x\in X}$ be the image of some tuple $(Q_y)_{y\in Y}$ with entries $Q_y\in R^y$.
Fix a point $x\in X$. 
The task is to verify $P_x\in R^x$. According to \eqref{eq:induced map}, the entry $P_x$ depends only on $Q_y$ for $y=f(x)$.
There is nothing to show if $x\in X$ is closed. Suppose now that $x$ is non-closed,  and fix
some  $\sigma=\sigma_i$ from $\Frk(x)=\{x,\sigma_1,\ldots,\sigma_r\}$.
Continuity ensures that  $\tau=f(\sigma)$ is  a specialization of $y$.
Recall that $Q_y$ is a rational function in the $T_V$, $V\ni y$
and $P_x$ is the corresponding rational function in the $T_{f^{-1}(V)}$.
In case $\tau=y$ we have $x\in f^{-1}(V)$, which immediately gives   $P_x\in R^x_\sigma$.
Suppose now that $\tau\neq y$,   choose an open set $V_y$ that contains $y$ but not $\tau$, and write   
$Q_y=\sum_{j=0}^n \beta_j T_{V_y}$.
The coefficients  $\beta_j$ are   rational functions
in $T_V$ and $T_{V'}/T_{V''}$, where $V$ ranges over the open neighborhoods of $\tau$,
and $V',V''$ run  over the open sets that contain $y$ but not $\sigma$. The preimages
$$
U_x=f^{-1}(V_y),\quad  U=f^{-1}(V),\quad U'=f^{-1}(V')\quadand U''=f^{-1}(U'').
$$
have the analogous properties with respect to $x,\sigma\in X$, which gives $P_x\in R^x_\sigma$.
This holds for all $\sigma=\sigma_i$, and thus $P_x\in R^x$.

It remains to show that $(P_x)_{x\in X}$ satisfies the conditions \eqref{eq:condition subring}, provided 
that  the analogous conditions hold for  $(Q_y)_{y\in Y}$. 
As in the previous paragraph  fix  $x,\sigma\in X$ and form the images $y,\tau\in Y$.
The task is to verify that the residue class $\overline{P_x}\in R^x_\sigma/\Rad(R^x_\sigma)$ coincides with $P_\sigma$.
For $\tau=y$ this is obvious,  because   $Q_y$ is a rational function in the $T_V$, $V\ni y$
and    both $P_x$ and $P_\sigma$ are the corresponding rational function in the  $T_{f^{-1}(V)}$.
Suppose now $\tau\neq y$, and write $Q_y=\sum_{j=0}^n \beta_j T_{V_y}$ as in the preceding paragraph, now with $\beta_0=Q_\sigma$.
Then $U_x=f^{-1}(V_y)$ is an open set that contains $x$ but not $\sigma$, whence $ T_{U_x}$
is a uniformizer in $R^x_\sigma$, and therefore 
$$
\overline{P_x} = \beta_0(T_{f^{-1}(V)},T_{f^{-1}(V')}/T_{f^{-1}(V'')}) = Q_\sigma(T_{f^{-1}(W)}) = P_\sigma,
$$
with $\beta_0=\beta_0(T_V,T_{V'}/T_{V''})$ and $Q_\sigma=Q_\sigma(T_W)$, $W\ni\sigma$.
\qed

\medskip
From the projections   $\pr_a:F(X)\ra k(T_U)_{U\ni a}$ we obtain  a canonical map
\begin{equation}
\label{eq:canonical map}
g_X:X\lra \Spec (R(X)),\quad a\longmapsto \Kernel(\pr_a|R(X)).
\end{equation} 
We now can formulate the main result of this paper:

\begin{theorem}
\mylabel{thm:natural homeomorphism}
For each finite Kolmogoroff space $X$, the above map is a homeomorphism. Moreover, the formation of the ring $R(X)$ is functorial
and the homeomorphism $g_X$ is natural, both  with respect to    continuous surjections $f:X\ra Y$ of finite Kolmogoroff spaces.
\end{theorem}

\proof
The formation of $R(X)$ is functorial in continuous surjections of finite Kolmogoroff spaces by Proposition \ref{prop:respects subrings}.
One easily sees that the diagram
$$
\begin{CD}
X^\dsc	@>g_X^\dsc>>	\Spec ( F(X))\\
@VVV		@VVV\\
X	@>>g_X>	\Spec( R(X))
\end{CD}
$$
is commutative, where $X^\dsc$ is the discretization of $X$, and the upper horizontal map $g_X^\dsc$ is defined in the same way as $g_X$.
The naturality of the latter is obvious. Since  $X^\dsc\ra X$ is surjective,   naturality of $g_X$ follows.

It remains to verify that $g_X$ is a homeomorphism. To this end consider the descending chain of closed 
sets $X=X_0\supset X_1\supset\ldots\supset X_n\supset X_{n+1}=\varnothing$,
where  $X_{i+1} =X_i\smallsetminus X_i^0$ is complement of the set of generic points $\eta\in X_i$,  and $n=\dim(X)$.
Consider the subrings  $R_i\subset\prod_{x\in X_i}R^x$ defined by the condition \eqref{eq:condition subring}, and the resulting maps
$$
g_i:X_i\lra \Spec(R_i),\quad a\longmapsto \Kernel(\pr_a|R_i),
$$
as in  \eqref{eq:canonical map}, which has  $g_X=g_0$. We now show by descending induction on $i\leq n$ that the $g_i$ are homeomorphisms. 
This is obvious for $i=n$, because the  space $X_n$ is discrete
and the ring $R_n$ is a  matching product of fields.  
Suppose now $i<n$, and that $g_{i+1}$ is a homeomorphism.
By definition, our rings sit in a cartesian square
$$
\begin{CD}
\prod_{\eta\in X_i^0} \overline{R^\eta}	@<<<	\prod_{\eta\in X_i^0} R^\eta\\
@AAA						@AAA\\
R_{i+1}					@<<<	R_i.
\end{CD}
$$
The upper horizontal map is surjective. We arrive at two cocartesian squares  
$$
\begin{CD}
\dot\bigcup \Frk^*(\eta)		@>>> 	\dot\bigcup  \Frk(\eta)\\
@VVV							@VVV\\
X_{i+1}						@>>>	X_i.
\end{CD}
\quadand
\begin{CD}
\Spec(\prod  \overline{R^\eta})	@>>>	\Spec(\prod  R^\eta)\\
@VVV						@VVV\\
\Spec(R_{i+1})					@>>>	\Spec(R_i)
\end{CD}
$$
of   spaces, the right by Lemma \ref{lem:cocartesian}, and the  left according to  \cite{Schwede 2005}, Theorem 3.3.
Here the unions and products run over all generic points $\eta\in X_i$.
The obvious identification 
$$
\dot\bigcup \Frk^*(\eta)=\Spec(\prod  \overline{R^\eta})\quadand \dot\bigcup  \Frk(\eta)=\Spec(\prod  R^\eta),
$$
together with our   $g_{i+1}$ and $g_i$,
define a map of squares. By our induction hypothesis, $g_{i+1}$ is a homeomorphism, and the universal property of cocartesian squares ensures
that  $g_i$ is a homeomorphism.
\qed

%===========================================================
\section{Spectralization functors and   Hochster's  Theorem}
\mylabel{sec:spectralization}

Let $X$ be a topological space. Consider the equivalence relation $x\sim x'$ given by the condition
that for every open set $U$ we have $x\in U\Leftrightarrow x'\in U$, and write $X^\kol$
for the ensuing quotient space. One easy checks that this  is a Kolmogoroff space,
and that the projection $X\ra X^\kol$ is universal for continuous maps to Kolmogoroff spaces.
The goal of this section is to construct in a related way  certain spectral spaces $X^\sprl$ and $X^\sprl_\qc$,
and examine their universal  properties.

Write  $\topO$ for  the topology on $X$,   and consider the collection 
of all coarser topologies $\topO_\lambda\subset\topO$, $\lambda\in L$ subject to  $|\topO_\lambda|<\infty$. 
We regard  $L$ as an  ordered set, where $\lambda\preccurlyeq \mu$ means $\topO_\lambda\subset\topO_\mu$.
The ordered set $L$ is filtered, because the topology generated by two topologies with only finitely many open sets
has only finitely many open sets.
The ensuing 
$X_\lambda = (|X|,\topO_\lambda)^\kol$
form a filtered  inverse system of Kolmogoroff spaces, and we define
$$
X^\sprl = \invlim_{\lambda\in L}X_\lambda.
$$
The quotient maps  $X\ra X_\lambda$ are compatible, and therefore induce a continuous map $X\ra X^\sprl$.
The construction is functorial: Given a continuous  mapping $f:X'\ra X$ we obtain   $f^*:L\ra L'$
by writing $f^{-1}(\topO_{\lambda})=\topO_{f^*(\lambda)}$, with resulting continuous mapping $f_\lambda:X'_{f^*(\lambda)}\ra X_\lambda$.
The ensuing compositions
$$
\prod_{\lambda'\in L'} X'_{\lambda'}\stackrel{\pr}{\lra} X'_{f^*(\lambda)} \stackrel{f_\lambda}{\lra} X_\lambda
$$
define a continuous map from the products of the $X'_{\lambda'}$ to the product of the $X_\lambda$, which induces
the desired map  continuous map $f^\sprl:X'^\sprl\ra X^\sprl$. One easily checks that the diagram
$$
\begin{CD}
X'		@>f>>		X\\
@VVV				@VVV\\
X'^\sprl	@>>f^\sprl>	X^\sprl
\end{CD}
$$
is commutative.

\begin{proposition}
\mylabel{prop:spectralization}
In the above situation, the   $X_\lambda$, $\lambda\in L$ are finite Kolmogoroff spaces, 
the transition maps $X_\mu\ra X_\lambda$, $\lambda\preccurlyeq \mu$ are surjective,   the inverse limit   $X^\sprl$ is a spectral space, and  
the map $X\ra X^\sprl$ has  dense image.
\end{proposition}

\proof
By construction, the $Y=X_\lambda$ are Kolmogoroff, and have only finitely many open sets. The latter ensures that each point $y$
has a smallest open neighborhood $V_y$,   the former tells us that the map $y\mapsto V_y$ is injective, and we conclude that $Y$ is finite.
The surjectivity of the projections $X\ra X_\lambda$ ensures that the transition maps are surjective. 
Given an irreducible closed set $Z\subset X^\sprl$, the images under projection have a unique generic point $\eta_\lambda$. One easily checks
that the family $\eta=(\eta_\lambda)$ is compatible, defines a  generic point of $Z$, and that there is no other generic point.
To see density of the image we fix a point  $(x_\lambda)\in X^\sprl$. The basic open neighborhoods
take the form 
$$
V=\pr_{\mu_1}^{-1}(V_{\mu_1})\cap\ldots\cap \pr_{\mu_r}^{-1}(V_{\mu_r})
$$
for some indices  $\mu_1,\ldots,\mu_r\in L$ and some  open neighborhoods $V_{\mu_i}$  of  $x_{\mu_i}\in X_\mu$.
Passing to  some $\mu\succcurlyeq \mu_1,\ldots,\mu_r$, we are reduced to the case $r=1$, and set $\mu=\mu_1$.
Then $V=\pr_{\mu}^{-1}(V_\mu)$  contains   image points, because $X\ra X_\mu$ is surjective.
\qed

\medskip
Now let  $X$ be a quasicompact space.  We   repeat the above  construction, but    
allow  only $\topO_\lambda$ whose members $U$ are quasicompact with respect to
the original topology $\topO$. Write $L_\qc\subset L$ for the corresponding set of indices,
and define
$$
X_\qc^\sprl = \invlim_{\lambda\in L_\qc} X_\lambda.
$$
As above, this defines a spectral space. 
One easily checks that the canonical continuous map  $X\ra X^\sprl_\qc$ is  quasicompact,
and that   the  construction is functorial for   continuous maps $f:X'\ra X$ that are quasicompact.

\begin{proposition}
\mylabel{prop:spectralization spectral space}
For each spectral space $X$, the ordered set $L_\qc$ is filtered, and the canonical map   $p:X\ra X^\sprl_\qc$ is a homeomorphism.
\end{proposition}

\proof
Given  $\lambda,\lambda'\in L_\qc$, let $\topO_\mu$ be the topology formed by  the intersections $U\cap U'$ 
with $U\in \topO_\lambda$ and $U'\in\topO_{\lambda'}$. These intersections are quasicompact, because $X$ is spectral,
hence $\mu\in L_\qc$, and the index set $L_\qc$ is filtered.

To proceed, we write $p_\lambda:X\ra X_\lambda$ for the quotient maps and  $p_{\mu\lambda}:X_\mu\ra X_\lambda$ for the 
transition maps. 
We first check that $p:X\ra X^\sprl_\qc$ is injective: Given points $a\neq b$, we find some quasicompact open set $U$ 
with $a\in U$ and $b\not\in U$, after swapping the points if necessary.   The  three open sets $ X,U,\varnothing$ constitute a topology  of the form $\topO_\mu$
for some $\mu\in L_\qc$, and $X_\lambda=\{\eta,\sigma\}$ can be identified with the  \emph{Sierpinski space}.
Then $p_\mu(a)=\eta\neq \sigma = p_\mu(b)$, hence $p$ is injective.

The main task is to verify that $p:X\ra X^\sprl_\qc$ surjective: Let $(z_\lambda)_{\lambda\in L_\qc}$ be a point from $X^\sprl_\qc$.
Then the  irreducible closed sets $Z_\lambda=\overline{\{z_\lambda\}}$ satisfy 
\begin{equation}
\label{eq:transition irreducible sets}
p_{\mu\lambda}(Z_\mu)\subset Z_\lambda\quadand p_{\mu\lambda}^{-1}(X_\lambda\smallsetminus Z_\lambda)\subset X_\mu\smallsetminus Z_\mu.
\end{equation} 
Consider the    quasicompact open sets $U_\lambda=p_\lambda^{-1}(X_\lambda\smallsetminus Z_\lambda)$ 
and the   closed set $Z=X\smallsetminus \bigcup U_\lambda$. 
We first observe that $Z$ is non-empty: Otherwise $X$ is covered by the $U_\lambda$,
and quasicompactness gives $X=U_{\lambda_1}\cup\ldots\cup U_{\lambda_r}$.
Using that $L_\qc$ is filtered together with \eqref{eq:transition irreducible sets},  we are reduced to the case $r=1$. Setting $\mu=\lambda_1$,
we see that $X$ equals the preimage of $X_\mu\smallsetminus Z_\mu$, in contradiction to the surjectivity of $X\ra X_\mu$.

We next check that $Z$ is irreducible: Otherwise there are two quasicompact open sets $V,V'\subset X$, both intersecting $Z$
but having  $V\cap V'\subset U$.  
The open sets $X,V,V',\varnothing$ constitute a topology  of the form $\topO_\lambda$ for some $\lambda\in L_\qc$,
and the resulting $X_\lambda=\{\eta,\zeta,\zeta',\sigma\}$ has
$$
p_\lambda^{-1}(\{\eta,\zeta\})=V\quadand p_\lambda^{-1}(\{\eta,\zeta'\})=V'\quadand p^{-1}_\lambda(\{\eta\})=V\cap V'.
$$
From this we infer  $z_\lambda=\eta$.
Using that $V\cap V'$ is quasicompact, we find some $\mu\succcurlyeq\lambda$ in $ L_\qc$ such that $V\cap V'\subset U_\mu$. Choose $x\in X$ with $p_\mu(x)=z_\mu$.
Then $x$ does not belong to $U_\mu=X\smallsetminus p_\mu^{-1}(Z_\mu)$,   therefore $x\not\in V\cap V'$, 
consequently $p_\lambda(x)\neq \eta =z_\lambda=p_{\mu\lambda}(z_\mu) =p_{\mu\lambda}(p_\mu(x))= p_\lambda(x)$, contradiction.

Since $X$ is spectral, the irreducible closed set $Z$ has a unique generic point $z$, 
and we indeed have  $p(z)=(z_\lambda)_{\lambda\in L_\qc}$.
To see this, fix an index $\lambda$.  Since $X_\lambda$ embeds into a finite product of Sierpinski spaces,
it suffices to treat the case that  $X_\lambda=\{\eta,\sigma\}$ equals the Sierpinski space, where the map $p_\lambda$ 
is already determined by the quasicompact open set $W=p_\lambda^{-1}(\eta)$. Choose $x\in X$ with $p_\lambda(x)=z_\lambda$.
If $x,z\in W$ or $x,z\not\in W$ we have $p_\lambda(z)=p_\lambda(x)=z_\lambda$. 
The case $z\in W, x\not\in W$ is impossible, because then $z_\lambda=p(x)=\sigma$ and thus  $W= p_\lambda^{-1}(X\smallsetminus\{z_\lambda\}) =U_\lambda\subset U$,
in contradiction to $z\in W$. It remains to rule out $z\not\in W, x\in W$, when $p_\lambda(x)=\eta$.
Using that $W$ is quasicompact, we find some $\mu\succcurlyeq\lambda$ in $ L_\qc$ such that $W\subset U_\mu$.
Choose $x'\in X$ with $p_\mu(x')=z_\mu$. Obviously, $x'$ is not contained in $U_\lambda=p_\lambda^{-1}(X_\lambda\smallsetminus Z_\lambda)$
and in particular not in $W$, and thus $p_\lambda(x')=\sigma$. On the other hand, we have
$p_\lambda(x')= p_{\mu\lambda}(p_\mu(x')) = p_{\mu\lambda}(z_\mu)=z_\lambda =p_\lambda(x)=\eta$,
contradiction.

Summing up, $p:X\ra X^\sprl_\qc$ is a continuous bijection, and it remains to check that  for each open $U\subset X$,
the image   is open. It suffices to treat the case that $U$ is quasicompact, and furthermore $U\neq X,\varnothing$. 
Then $X,U,\varnothing$ constitute a topology
of the form $\topO_\lambda$ for some $\lambda\in L_\qc$,     the  resulting $X_\lambda=\{\eta,\sigma\}$ is a copy of the Sierpinski space,
and $U=p_\lambda^{-1}(\eta)$.
Surjectivity of $p$ gives  $ p(p_\lambda^{-1}(\eta))=p_\lambda^{-1}(\eta)$, which is open.
\qed

\medskip
In different form, the  following universality property already appears in  
\cite{Banaschewski 1996}, Proposition 4 and  \cite{Schwartz 2013}, Theorem 5.1:

\begin{theorem}
\mylabel{thm:universal map}
For each quasicompact space $X$, the map   $X\ra X^\sprl_\qc$ is universal for  quasicompact continuous maps to spectral spaces.
\end{theorem}

\proof
Let $Y$ be a spectral space and  $f:X\ra Y$ be  quasicompact and continuous. We then have a commutative diagram
$$
\begin{tikzcd} 
X\ar[r]\ar[d,"f"']     	 &   X^\sprl\ar[d]\ar[r,dashed]	& X^\sprl_\qc\ar[d,dashed]\\
Y\ar[r]                  &   Y^\sprl\ar[r]		& Y^\sprl_\qc
\end{tikzcd}
$$
The desired factorization of $f$ over $X^\sprl_\qc$ is given applying  the vertical map on the right, followed
by  the inverse of the lower composition. 
Note that in the diagram, the undashed arrows exists in general, while the dashed arrows rely on quasicompactness of 
the space  $X$ and the map $f:X\ra Y$.

It remains to verify uniqueness. In light of the identification $Y=\invlim Y_\lambda$ it suffices to treat the case
that $Y$ is finite. Since a finite Kolmogoroff spaces sits inside a finite product of the Sierpinski space $S=\{\eta,\sigma\}$,
we may assume that $Y$ itself equals the Sierpinski space. Now continuous maps to $Y$ can be identified with open sets  in the domain,
and our task is to show that two quasicompact open sets $V,W\subset X^\sprl_\qc$ whose preimages 
on $X$ coincide already coincide.
Write 
\begin{equation}
\label{eq:quasicompact product topology}
V=\bigcup_{i=1}^r\pr_{\alpha_i}^{-1}(V_i)\quadand W=\bigcup_{j=1}^s\pr_{\beta_j}^{-1}(W_j)
\end{equation} 
for certain indices $\alpha_i,\beta_j\in L_\qc$ and open sets $V_i\subset X_{\alpha_i}$ and $W_j\subset X_{\beta_j}$.
Choose some $\lambda\in L_\qc$ so that $\topO_\lambda$ is finer than all the $\topO_{\alpha_i}$ and $\topO_{\beta_j}$.
Then the common preimage $U\subset X$ of \eqref{eq:quasicompact product topology} belongs to $\topO_\lambda$, and thus corresponds to an 
 open set $U_\lambda\subset X_\lambda$. 
By construction, both open sets in \eqref{eq:quasicompact product topology} coincide with $\pr^{-1}_\lambda(U_\lambda)$.
\qed

\medskip
Now fix a ground field $k$.  For each spectral space $X$, the  filtered inverse system $(X_\lambda)_{\lambda\in L_\qc}$
of finite Kolmogoroff spaces with surjective transition maps yields a filtered direct system $(R_\lambda)_{\lambda\in L_\qc}$
of prime-finite rings $R_\lambda=R(X_\lambda)$,  coming with compatible  homeomorphisms
$X_\lambda\ra \Spec(R_\lambda)$, as constructed in Section \ref{sec:functorial construction}.
We thus get continuous maps
$$
X\lra   \invlim X_\lambda \lra \invlim \Spec (R_\lambda) \lra  \Spec(\dirlim R_\lambda).
$$
where the indices run over $L_\qc$. The map in the middle is a homeomorphism (Theorem \ref{thm:natural homeomorphism}). The same holds for
the map on the right (\cite{EGA IVc}, Corollary 8.2.10) and the map on the left (Proposition \ref{prop:spectralization spectral space}).  
This gives  the desired  simple proof of Hochster's theorem (\cite{Hochster 1969}, compare also   \cite{Dickmann; Schartz; Tressl 2019}, Section 12.6 and \cite{Ershov 2005}):

\begin{theorem}
\mylabel{thm:hochster}
Each spectral space  is homeomorphic to the spectrum of a ring.
\end{theorem}

\medskip
Note that the ring $R(X)=\dirlim R(X_\lambda)$ is canonically constructed out of nothing but the spectral space $X$ and the   ground field $k$.
One may wonder whether  there are any ties with reconstructions of algebraic schemes from the Zariski topology 
\cite{Kollar; Lieblich; Olsson; Sawin 2023},
or what happens if one works  over other ground rings rather than ground fields, in particular with the ring of integers $\ZZ$.

%===========================================================

\end{document}